\documentclass[11pt]{amsart}
\usepackage{amssymb}
\usepackage{verbatim}
\usepackage{times} 

\begin{document}

\newtheorem{theorem}{Theorem}    
\newtheorem{proposition}[theorem]{Proposition}
\newtheorem{conjecture}[theorem]{Conjecture}
\def\theconjecture{\unskip}
\newtheorem{corollary}[theorem]{Corollary}
\newtheorem{lemma}[theorem]{Lemma}
\newtheorem{sublemma}[theorem]{Sublemma}
\newtheorem{observation}[theorem]{Observation}
\theoremstyle{definition}
\newtheorem{definition}{Definition}
\newtheorem{notation}[definition]{Notation}
\newtheorem{remark}[definition]{Remark}
\newtheorem{question}[definition]{Question}
\newtheorem{example}[definition]{Example}
\newtheorem{problem}[definition]{Problem}
\newtheorem{exercise}[definition]{Exercise}

\numberwithin{theorem}{section}
\numberwithin{definition}{section}
\numberwithin{equation}{section}

\def\earrow{{\mathbf e}}
\def\rarrow{{\mathbf r}}
\def\uarrow{{\mathbf u}}
\def\tpar{T_{\rm par}}
\def\apar{A_{\rm par}}

\def\reals{{\mathbb R}}
\def\torus{{\mathbb T}}
\def\heis{{\mathbb H}}
\def\integers{{\mathbb Z}}
\def\naturals{{\mathbb N}}
\def\complex{{\mathbb C}\/}
\def\distance{\operatorname{distance}\,}
\def\support{\operatorname{support}\,}
\def\dist{\operatorname{dist}\,}
\def\Span{\operatorname{span}\,}
\def\degree{\operatorname{degree}\,}
\def\kernel{\operatorname{kernel}\,}
\def\dim{\operatorname{dim}\,}
\def\codim{\operatorname{codim}}
\def\trace{\operatorname{trace\,}}
\def\Span{\operatorname{span}\,}
\def\ZZ{ {\mathbb Z} }
\def\p{\partial}
\def\rp{{ ^{-1} }}
\def\Re{\operatorname{Re\,} }
\def\Im{\operatorname{Im\,} }
\def\ov{\overline}
\def\eps{\varepsilon}
\def\lt{L^2}
\def\diver{\operatorname{div}}
\def\curl{\operatorname{curl}}
\def\etta{\eta}
\newcommand{\norm}[1]{ \|  #1 \|}
\def\Span{\operatorname{span}}
\def\expect{\mathbb E}

\newcommand{\Norm}[1]{ \left\|  #1 \right\| }
\newcommand{\set}[1]{ \left\{ #1 \right\} }
\def\one{\mathbf 1}
\newcommand{\modulo}[2]{[#1]_{#2}}

\def\scriptf{{\mathcal F}}
\def\scriptg{{\mathcal G}}
\def\scriptm{{\mathcal M}}
\def\scriptb{{\mathcal B}}
\def\scriptc{{\mathcal C}}
\def\scriptt{{\mathcal T}}
\def\scripti{{\mathcal I}}
\def\scripte{{\mathcal E}}
\def\scriptv{{\mathcal V}}
\def\scriptS{{\mathcal S}}
\def\frakv{{\mathfrak V}}

\author{Michael Christ}
\address{
        Michael Christ\\
        Department of Mathematics\\
        University of California \\
        Berkeley, CA 94720-3840, USA}
\email{mchrist@math.berkeley.edu}
\thanks{The author was supported in part by NSF grant DMS-0401260}

\date{May 28, 2009}

\title[A Weak Type $(1,1)$ Inequality]
{A Weak Type $(1,1)$ Inequality \\ For Maximal Averages   Over Certain Sparse Sequences}

\maketitle


\section{Introduction}

To any strictly increasing sequence of nonnegative 
integers $(n_\nu: \nu\in\naturals)$
is associated a maximal operator
$M$, which maps functions $f\in\ell^1(\integers)$
to $Mf:\integers\to[0,+\infty]$, defined by
\begin{equation}
Mf(x) = \sup_N N^{-1}\Big|\sum_{\nu=1}^N f(x+n_\nu)\Big|.
\end{equation}
The most fundamental example is the Hardy-Littlewood maximal
function for $\integers$, for which $n_\nu\equiv \nu$.
In this note we construct sequences which satisfy
\begin{equation}
n_\nu\asymp \nu^m
\end{equation}
for arbitrary integer exponents $m\ge 2$,
and which have a certain algebraic character,
for which the associated maximal operator is of weak type $(1,1)$.

\medskip

\begin{notation}
For any positive integer $N$ and any $n\in\integers$,
$\modulo{n}{N}$ denotes the unique element
of $\{0,1,\cdots,N-1\}$ congruent to $n$ modulo $N$.
For $n=(n_1,\cdots,n_d)\in\integers^d$,
$\modulo{n}{N} = (\modulo{n_1}{N},\cdots,\modulo{n_d}{N})$. 

If $n_\nu,c_\nu$ are sequences of positive integers which tend to
infinity, we write $n_\nu\asymp c_\nu$ to indicate that
the ratio $\frac{n_\nu}{c_\nu}$ is bounded above and below by strictly
positive finite constants, independent of $\nu\in\naturals$.
\end{notation}

Let $(p_k)$ be a lacunary sequence of primes which satisfies
\begin{equation}\label{primehypotheses}
\left.
\begin{aligned}p_{k+1}&\ge (1+\delta)p_k 
\\
p_{k+1}&\le Cp_k
\end{aligned}
\right\}
\text{ for all } k
\end{equation}
for some $\delta>0$ and $C<\infty$.
Let $m\ge 2$ be a positive integer.
Let $(a_k)$ be an auxiliary sequence of positive integers satisfying 
\begin{align}
\label{secondajcondition}
a_k &\le Cp_k^m
\\
a_{k+1}&> a_k + p_k^m.
\label{firstajcondition}
\end{align}
Define
$\scriptS_k\subset\integers$ to be
\begin{equation}
\scriptS_k=\big\{
a_k + \sum_{r=1}^m p_k^{r-1} \modulo{j^r}{p_k}:
0\le j<p_k\big\}.
\end{equation}
Thus $\scriptS_k$ has cardinality
\begin{equation}
|\scriptS_k| = p_k
\end{equation}
and
$\scriptS_k\subset \big[a_k, a_k + p_k^m\big]$.
Let $\scriptS=\cup_{k=1}^\infty \scriptS_k$,
and define the sequence $(n_\nu: \nu\in\naturals)$ 
to be the elements of $\scriptS$, listed in increasing order.
The condition \eqref{firstajcondition} ensures that
every element of $\scriptS_{k+1}$ is strictly
greater than every element of $\scriptS_k$.

\begin{theorem}
\label{thm:main}
Let $m\ge 2$ be an integer.
Let $(n_\nu: \nu\in\naturals)$ be the subsequence of $\naturals$
constructed via the above recipe from a lacunary sequence of primes $p_k$
and a sequence of positive integers $a_k$
satisfying \eqref{primehypotheses}, 
\eqref{firstajcondition}, and \eqref{secondajcondition}.
Then $n_\nu\asymp\nu^m$, and the maximal function associated to $(n_\nu)$ is of weak type $(1,1)$
on $\integers$.
\end{theorem}

These sequences are closely related to examples given by Rudin
\cite{rudin} of $\Lambda(p)$ sets for $p=4,6,8,\cdots$.

Bourgain \cite{bourgain1},\cite{bourgain2},\cite{bourgain3} proved that
the maximal operator associated to the sequence $a_\nu = \nu^m$
is bounded on $\ell^q(\integers)$ for all $q>1$, for arbitrary $m\in\naturals$,
but the situation for the endpoint $q=1$ was left unresolved.
There have recently been several works concerned with weak type $(1,1)$ inequalities
for maximal operators associated to sparse subsequences of integers.
Buczolich and Mauldin \cite{bm1},\cite{bm2}
have shown that the maximal operator associated
to the sequence of all squares $(\nu^2:\nu\in\integers)$
is not of weak type $(1,1)$. 
LaVictore \cite{patrickbm}
has extended their method to show that
the same holds for $(n^m: n\in\integers)$
for all positive integer exponents $m$.
In the positive direction, Urban and Zienkiewicz \cite{uz}
have shown that for any real exponent $\alpha>1$ 
sufficiently close to $1$,
the maximal function associated to the sequence $\lfloor{\nu^\alpha}\rfloor$
is of weak type $(1,1)$. 
LaVictoire \cite{patrick1} has shown that certain random sequences
satisfying $n_\nu\asymp\nu^m$ 
almost surely give rise to maximal operators which are of
weak type $(1,1)$ for arbitrary real exponents $m\in (1,2)$;
it remains an open question whether the conclusion holds for these random sequences
when $m\in[2,\infty)$.

Let $\sigma$ denote
surface measure on the unit sphere $S^{d-1}\subset\reals^d$.
It remains an open question whether
the maximal operator $Mf(x) = \sup_{k\in\integers}
\int_{S^{d-1}}|f(x-2^k y)|\,d\sigma(y)$ is of weak type $(1,1)$
in $\reals^d$ for $d\ge 2$. Discrete analogues of continuum problems
are often more delicate, but
our analysis will exploit two discrete phenomena which lack
obvious continuum analogues.

$f\mapsto Mf$ is of weak type $(1,1)$ if and only if
the same goes for $f\mapsto M(|f|)$, so it suffices to restrict
attention to nonnegative functions.
For any strictly increasing sequence $(n_\nu)$,
\[c\sup_k |f*\mu_k| \le Mf\le C\sup_k |f*\mu_k|\]
for all nonnegative functions $f$,
where $\mu_k$ is the measure
$\mu_k = 2^{-k} \sum_{\nu=2^k+1}^{2^{k+1}} \delta_{n_\nu}$,
and $\delta_n$ denotes the Dirac mass at $n$.
Therefore $f\mapsto Mf$ is of weak type $(1,1)$, if and only
if the same goes for $f\mapsto \sup_k |f*\mu_k|$.

The following is a sufficient condition, of Tauberian type, for such a maximal
operator to be of weak type $(1,1)$.
Although our application will be to operators on $\ell^1(\integers)$,
this result makes sense for any discrete group and is no more complicated
to prove in that setting, so we give the general formulation.
\begin{theorem}
\label{thm:difference}
Let $G$ be a discrete group.
Let $\gamma>0$.
Let $\mu_k,\nu_k:G\to\complex$ satisfy
\begin{equation}
\label{hyp:numaxfn}
\text{The maximal operator $\sup_k |f|*|\nu_k|$ is of weak type $(1,1)$ on $G$,}
\end{equation}
each $\mu_k$ satisfies
\begin{equation}
|\support(\mu_k)|\le C2^{k\gamma},
\end{equation}
and
\begin{equation} \label{elltwohypothesis}
\norm{f*(\mu_k-\nu_k)}_{\ell^2} 
\le C2^{-k\gamma/2}\norm{f}_{\ell^2}
\ \ \ \forall f\in \ell^2(\integers^d).
\end{equation}
Then the maximal operator $\sup_{k\in\naturals} |f*\mu_k|$ is of weak type $(1,1)$ on $G$.
\end{theorem}

For $G=\integers^d$,
the last hypothesis can be equivalently restated as
\begin{equation}
\norm{\widehat{\mu_k}-\widehat{\nu_k}}_{L^\infty(\torus^d)}\le C2^{-k\gamma/2},
\end{equation}
where $\torus = \reals/\integers$.
In our applications, 
$\mu_k$ will be a probability measure and hence $\widehat{\mu_k}(0)$ cannot be small.
$\nu_k$ will be a simpler measure, constructed in order to correct $\widehat{\mu_k}(\theta)$
for small $\theta$.

\medskip
The author is indebted to Steve Wainger for generous and essential advice 
concerning trigonometric sums, and for supplying reference \cite{LN}, 
and to Patrick LaVictoire for advice concerning the exposition.

\section{Analysis of maximal operators}

\begin{notation}
Let $G$ be a discrete group.
For any subsets $A,B\subset G$,
$A+B=\{ab: a\in A \text{ and } b\in B\}$,
where $ab$ denotes the product of two group elements.\footnote{The additive notation
is used for $A+B$, even though $G$ is not assumed to be Abelian, in order to
simplify an expression below.}
\end{notation}

There is the simple inequality
\begin{equation}\label{sumbound} |A+B|\le |A|\cdot|B|,\end{equation}
which has no straightforward analogue in continuum situations. 
Following Urban and Zienkiewicz \cite{uz}, we will make essential use of 
\eqref{sumbound}. 

\begin{proof}[Proof of Theorem~\ref{thm:difference}]
Let $f\in \ell^1$ and let $\alpha>0$.
We seek an upper bound for $|\{x: \sup_k |f*\mu_k(x)|>\alpha\}|$. 
Decompose $f = \sum_{j=-\infty}^\infty f_j$ where 
\[
f_j(x)=
\begin{cases}
f(x)
& \text{ if } |f(x)|\in [2^j,2^{j+1}),
\\
0 & \text{ otherwise.}
\end{cases}
\]
The functions $f_j$ have pairwise disjoint supports, so 
$\sum_j\norm{f_j}_{\ell^1} = \norm{f}_{\ell^1}$.

For each $j\in\integers$ define the exceptional set  
\begin{equation}
\scripte_j = 
\cup_{2^{k\gamma}<\alpha^{-1}2^j}
\big(\support(f_j)+\support(\mu_k) \big)
\end{equation}
and
\begin{equation}
\scripte = \cup_{j\in\integers} \scripte_j.
\end{equation}
Since $\norm{f_j}_1\sim 2^j|\support(f_j)|$,
\begin{equation}
|\scripte_j| \le
\sum_{k: 2^{k\gamma}<\alpha^{-1}2^j}
|\support(f_j)|\cdot 2^{k\gamma}
\le
2\sum_{k: 2^{k\gamma}<\alpha^{-1}2^j}
2^{-j} 2^{k\gamma} \norm{f_j}_1
\le C_\gamma\alpha^{-1}\norm{f_j}_1
\end{equation}
and therefore
\begin{equation}	
|\scripte|\le C\alpha^{-1}\norm{f}_1.
\end{equation}

Set $\lambda_k = \mu_k-\nu_k$.
For $x\notin\scripte$, 
\begin{equation}
f*\mu_k(x) = 
\sum_{2^j\le 2^{k\gamma}\alpha}
f_j*\mu_k(x)
\end{equation}
so that
\begin{align*}
|f*\mu_k(x)| 
&\le 
\sum_{2^j\le 2^{k\gamma}\alpha}
|f_j|*|\nu_k|(x)
+
\sum_{2^j\le 2^{k\gamma}\alpha}
|f_j*\lambda_k(x)|
\\
&\le
|f|*|\nu_k|(x)
+
\sum_{2^j\le 2^{k\gamma}\alpha} |f_j*\lambda_k(x)|.
\end{align*}

Therefore
\begin{multline}
\big|\{x\in G: \sup_k|f*\mu_k(x)|>\alpha\}\big|
\\
\le |\scripte| + 
\big|\{x\in G: \sup_k|f|*|\nu_k|(x)>\tfrac\alpha2\}\big|
+ 
\big|\{x\in G: 
\sup_k\sum_{2^j\le 2^{k\gamma}\alpha} |f_j*\lambda_k(x)|
>\tfrac\alpha2\}\big|.
\end{multline}
Therefore, by hypothesis \eqref{hyp:numaxfn}, it suffices to show that
\begin{equation}
|\{x\in G:
\sup_k
\sum_{j:2^j\le 2^{k\gamma}\alpha}
|f_j*\lambda_k(x)|
>\tfrac\alpha2
\}|
\le C\alpha^{-1}\norm{f}_1.
\end{equation}

For $0\le s\in\integers$ define 
\begin{equation}
G_s(x) = 
\big(
\sum_k 
|f_{j(k,s)}*\lambda_k(x)|^2
\big)^{1/2}
\end{equation}
where $j(k,s)$ is the unique integer satisfying
$2^{k\gamma-s}\alpha\le 2^{j(k,s)}< 2^{k\gamma-s+1}\alpha$. 
A generous upper bound is
\begin{equation}
\sup_k 
\sum_{j:2^j\le 2^{k\gamma}\alpha}
|f_j*\lambda_k(x)|
\le
\sum_{s=0}^\infty G_s(x)
\end{equation}

Therefore by hypothesis \eqref{elltwohypothesis},
\begin{align*}
\norm{G_s}_2^2 
&= \sum_k \norm{f_{j(k,s)}*\lambda_k}_2^2
\le C\sum_k 2^{-k\gamma}\norm{f_{j(k,s)}}_2^2
\\
&\ \ \le C\sum_k 2^{-k\gamma}\norm{f_{j(k,s)}}_\infty\norm{f_{j(k,s)}}_1
\le C\sum_k 2^{-k\gamma} 2^{j(k,s)}\norm{f_{j(k,s)}}_1
\\
&\ \ \ \ \le C\sum_k 2^{-k\gamma} 2^{k\gamma-s}\alpha\norm{f_{j(k,s)}}_1
= C 2^{-s}\alpha\sum_k\norm{f_{j(k,s)}}_1
\ \ 
\le C 2^{-s}\alpha\norm{f}_1.
\end{align*}
Therefore
$\norm{\sum_s G_s}_2^2 \le C\alpha \norm{f}_1$
and consequently
\begin{equation}
|\{x: \sum_{s=0}^\infty G_s(x)>\tfrac\alpha2\}|
\le 4\alpha^{-2}\norm{\sum_s G_s}_2^2
\le C\alpha^{1-2}\norm{f}_1
= C\alpha^{-1}\norm{f}_1.
\end{equation}
This concludes the proof of Theorem~\ref{thm:difference}.
\end{proof}

The use of an $L^2$ bound on the complement of an exceptional set in order to
obtain a weak type $(1,1)$ inequality was pioneered by Fefferman \cite{cfthesis}, and
was applied to maximal functions in \cite{christrough}.
Exceptional sets constructed as algebraic sums of supports, adapted to the measures
$\mu_k$, were used for a continuum analogue of this problem in Theorems 3 and 4
of \cite{christrough}.

\section{Construction of examples}

\subsection{The Fourier transform on a finite cyclic group.}
For any positive integer $N$
let $\integers_N = \integers/N\integers$
be the cyclic group of order $N$.
We will often identify elements of $\integers_N$
with elements of $[0,1,\cdots,N-1]$ in the natural way.

The dual group of $\integers_{p}^m$ may, and will, 
be identified with $\integers_{p}^m$. The most convenient normalization
of the Fourier transform for our purposes is
\begin{equation}
\widehat{f}(\xi) = \sum_{k} f(k)e^{-2\pi i k\cdot\xi/p}
\qquad\text{for } \xi\in\integers_p^m
\end{equation}
where $k\cdot\xi$ denotes the usual Euclidean inner product
of two elements of $[0,1,\cdots,p-1]^m$, regarded as elements
of $\integers^m$.
This Fourier transform satisfies
\begin{equation}
\norm{\widehat{f}}_{\ell^2} = p^{m/2}\norm{f}_{\ell^2}
\qquad\text{and}\qquad
f(k) = p^{-m}\sum_{\xi} \widehat{f}(\xi)e^{2\pi i k\cdot \xi/p}.
\end{equation}

The convolution of two functions is
$f*g(x) = \sum_{y}f(x-y)g(y)$.
Products and convolutions are related by
\begin{equation}
\widehat{f*g} = \widehat{f}\cdot\widehat{g}
\qquad\text{and}\qquad
\widehat{fg} = p^{-m}  \widehat{f}*\widehat{g}.
\end{equation}
Therefore 
\begin{equation} \label{convolveFTs}
\norm{\widehat{fg}}_\infty
\le  p^{-m}\norm{\widehat{f}}_1\norm{\widehat{g}}_\infty
\end{equation}
and
\begin{equation}
\norm{f*g}_2 = p^{-m/2}\norm{\widehat{f}\cdot\widehat{g}}_2
\le p^{-m/2}\norm{\widehat{g}}_\infty\norm{\widehat{f}}_2
= \norm{\widehat{g}}_\infty\norm{f}_2.
\end{equation}

\subsection{On certain exponential sums.}
Define the probability measure $\sigma_{p,m}$
on $\integers_{p}^m$ to be
\begin{equation}
\sigma_{p,m} = p^{-1}\sum_{k=0}^{p-1} \delta_{(k,k^2,\cdots,k^m)}.
\end{equation}
Thus 
\begin{equation}
\big|\support(\sigma_{p,m})\big|=p=|\integers_{p}^m|^{1/m}.
\end{equation}
This will lead to examples of Theorem~\ref{thm:difference}
with exponent $\gamma=\frac1m$. 

\begin{lemma}[Weil] \label{lemma:weil}
Let $m\ge 1$, and
let $p>m$ be prime. 
Then 
\begin{equation}
|\widehat{\sigma_{p,m}}(\xi)| \le (m-1)p^{-\frac12}
\ \ \text{ for all $0\ne\xi\in \integers_p^m$.}
\end{equation}
\end{lemma}

Three comments are in order. Firstly, this illustrates a general principle
that $\integers_p$ has only one scale when $p$ is prime. 
In contrast,
the most natural example of a sparsely supported measure on $\reals^d$
whose Fourier transform exhibits power law decay is surface measure $\sigma$ on the unit
sphere $S^{d-1}$ for $d\ge 2$, 
which satisfies $|\widehat{\sigma}(\xi)|\sim |\xi|^{-(d-1)/2}$ for generic large $\xi$;
the size of $\widehat{\sigma}$ is best described for most $\xi$
by a power of $|\xi|$, rather than by a constant. 

Secondly, no weaker bound $O(p^{-\tfrac12+\delta})$ will suffice in the construction below
to yield a sequence satisfying the hypotheses of Theorem~\ref{thm:difference}.

Thirdly, the bound $O(p^{-1/2})$ is the best that can hold for a measure
$\sigma$ whose support  has cardinality $p$, unless $\norm{\sigma}_{\ell^1}\ll 1$.
Indeed, let $E$ be the support of $\sigma$.
Let $\sigma_0$ be the constant function
$\sigma_0(n)=cp^{-m}$ for all $n\in\integers_p^m$,
where $c = \sum_n \sigma(n)$.
Then
$\widehat{\sigma-\sigma_0}(\xi)$ vanishes at $\xi=0$,
and $=\widehat{\sigma}(\xi)$ otherwise.
Consequently 
\begin{multline*}
\norm{\sigma-\sigma_0}_{\ell^1(E)}
\le |E|^{1/2}  
\norm{\sigma-\sigma_0}_{\ell^2}^{1/2}
= 
p^{1/2} p^{-m/2} \norm{\widehat{\sigma-\sigma_0}}_{\ell^2}
\\
\le p^{1/2} p^{-m/2} |\integers_p^m|^{1/2}
\norm{\widehat{\sigma}-\widehat{\sigma_0}}_{\ell^\infty}^{1/2}
= p^{1/2} 
\norm{\widehat{\sigma}-\widehat{\sigma_0}}_{\ell^\infty}^{1/2}.
\end{multline*}
Since $\norm{\sigma_0}_{\ell^1(E)}\le p^{-m}\norm{\sigma}_1|E|$,
this implies that
\[
\norm{\sigma}_1=\norm{\sigma}_{\ell^1(E)}
\le 
p^{1/2} \sup_{\xi\ne 0}|\widehat{\sigma}(\xi)|^{1/2}
+ p^{1-m}\norm{\sigma}_1,\]
so
\[
\norm{\sigma}_1
\le 
2p^{1/2} \sup_{\xi\ne 0}|\widehat{\sigma}(\xi)|^{1/2}.
\]
Thus the construction is tightly constrained. 


\begin{proof}[Proof of Lemma~\ref{lemma:weil}]
\[
\widehat{\sigma_{p,m}}(\xi)
= p^{-1}\sum_{k=0}^{p-1} e^{-2\pi i (k\xi_1+k^2\xi_2+\cdots+k^m\xi_m)/p}.
\]
The sum, without the initial factor $p^{-1}$, 
is a well studied quantity whose
absolute value is $\le (m-1)p^{1/2}$.
An elementary proof may be found in \cite{LN}, Theorem 5.38.
\end{proof}


Denote by $\delta_k$ the function $\delta_k(n)=1$ if $n=k$
and $=0$ if $n\ne k$.
The measures
$\sigma^0_{p,m}
= |\integers_{p}^m|^{-1}\sum_{k\in\integers_{p}^m}\delta_k$
satisfy 
\[
\widehat{\sigma^0_{p,m}}(\xi) = 
\begin{cases} 1  = \widehat{\sigma_{p,m}}(0) & \text{ if } \xi=0
\\
0 & \text{ else}
\end{cases}
\]
and therefore 
$\sigma_{p,m}^* = \sigma_{p,m}-\sigma^0_{p,m}$
satisfies
\begin{equation}
\norm{\widehat{\sigma_{p,m}^*}}_\infty
\le (m-1)p^{-1/2}.
\end{equation}

\subsection{Transference to $\integers^m$.}
We wish to transfer $\sigma_{p,m}^*$ to a measure on $\integers$, 
preserving this $L^\infty$ Fourier transform bound, 
in order to obtain the desired examples. 
The most straightforward attempt apparently does not work, but 
the following more roundabout procedure, combining an extension
to $\integers_{3p}^m$ with cutoff functions, does the job.
It will be convenient to transfer first to $\integers^m$,
then to $\integers$ in a separate step.

Consider $\integers_{3p}^m$, which we identify with
$[-p,2p-1]^m$. Likewise we identify $\integers_{3p}$
with $[-p,2p-1]$.
Assume that $p$ is odd.
The function 
\[
\kappa(k)=
\sigma_{p,m}^*(\modulo{k}{p}) 
\]
is $p$-periodic on $\integers_{3p}$ and satisfies
\begin{equation}
\widehat{\kappa}(\xi)
= 
\begin{cases}
0 & \text{ unless each $\xi_j\in[-p,2p-1]$ is divisible by $3$}
\\
3^m\widehat{\sigma_{p,m}^*}(\xi_1/3,\cdots,\xi_m/3) & \text{ otherwise.}
\end{cases}
\end{equation}
In particular,
\begin{equation}
\norm{\widehat{\kappa}}_\infty \le Cp^{-1/2}.
\end{equation}

Define $\varphi:\integers_{3p}\to\reals$ by
\begin{equation}
\begin{cases}
\varphi(i) = 
1 & \text{ if } i\in [0,p-1]
\\
\varphi(i) = 
0 & \text{ if } i\in [-p, -p + (p-1)/2]
\\
\varphi(i) = 
0 & \text{ if } i\in [p-1+ (p-1)/2, 2p-1]
\\
\varphi \text{ is affine} & \text{ on the interval } 
[-p+(p-1)/2,0]
\\
\varphi \text{ is affine} & \text{ on the interval } 
[p-1,p-1+(p-1)/2].
\end{cases}
\end{equation}
Define 
\[\phi(k_1,\cdots,k_m) = \prod_{i=1}^m \varphi(k_i).\]
Define $\rho=\rho_{p,m}:\integers_{3p}^m\to\reals$ by
\begin{equation}
\rho = \phi\kappa.
\end{equation}
$\rho$ is nonnegative, and
\begin{equation}
\rho(k)\equiv \sigma^*_{p,m}(k) \qquad\forall k\in [0,p-1].
\end{equation}

Define $\rho^\ddagger: \integers^m\to\reals$ by
\begin{equation}
\rho^\ddagger(k) = 
\begin{cases}
\rho(\modulo{k}{3p}) & \text{ for } k\in [-p,2p-1]^m
\\
0 & \text{ else}
\end{cases}
\end{equation}
where $k$ is interpreted as an element of $\integers^m$ on the left-hand
side, and as an element of $\integers_{3p}^m$ on the right.

\begin{lemma}
\label{lemma:sumbyparts2}
\begin{equation}
\norm{\widehat{\rho^\ddagger}}_{L^\infty(\torus^m)}
\le Cp^{-1/2}.
\end{equation}
\end{lemma}

\begin{proof}
Let $\theta\in\torus^m=[0,1]^m$. 
Write $\theta = \xi/3p+\eta$ 
where $\xi\in\integers^m$ and
$\eta=(\eta_1,\cdots,\eta_m)\in\reals^m$ satisfies $|\eta_j|\le Cp^{-1}$
for all $j$.
\begin{align*}
\widehat{\rho^\ddagger}(\theta)
&= \sum_{k\in\integers^m} \rho^\ddagger(k)e^{-2\pi ik\cdot\theta}
\\
&= \sum_{k\in [-p,2p-1]^m} \phi(k)\kappa(k)
e^{-2\pi i k\cdot\eta} 
e^{-2\pi ik\cdot\xi/3p}.
\end{align*}
Interpret this last expression as the Fourier transform on the group
$\integers_{3p}^m$, evaluated at $\xi$, of $\psi\kappa$
where
$\psi:\integers_{3p}^m\to\reals$
is defined by
$\psi(k) = \phi(k) e^{-2\pi i k\cdot\eta}$.
By \eqref{convolveFTs},
since $\widehat{\kappa}=O(p^{-1/2})$,
it suffices to show that 
$\norm{\widehat{\psi}}_{\ell^1} \le Cp^m$.

$\psi(k)$ factors as $\prod_{j=0}^m \varphi(k_j)
e^{-2\pi i k_j\cdot\eta_j}$, so its Fourier transform likewise factors.
Thus it suffices to prove that
\begin{equation}
\label{eq:psiFT}
\sum_{\xi=0}^{3p-1}\big|
\sum_{n=-p}^{2p-1} \varphi(n)e^{-2\pi i n\eps}
e^{-2\pi i \xi\cdot n/3p}\big|
\lesssim p
\qquad\text{provided that $|\eps|\le p^{-1}$.}
\end{equation}

For $\xi=0$ there is the trivial bound $O(p)$,
since $\norm{\varphi}_\infty = 1$.
For $\xi\in [1,\cdots,3p-1]$,
we employ the summation by parts formula
\[
\sum_{n=-p}^{2p-1} a_n b_n
=b_{2p}A_{2p-1} -\sum_{n=-p}^{2p-1} A_n\Delta b_n
\ \ \text{ where $A_n = \sum_{j=-p}^n a_j$
and $\Delta b_n = b_{n+1}-b_n$.}
\]

Define $\Phi(n) = \varphi(n)e^{-2\pi i n\eps}$.
For convenience of notation,
extend $\Phi$ to be a $3p$-periodic function on $\integers$.
Sum by parts with $a_n = e^{-2\pi i n\xi/3p}$
and $b_n = \Phi(n)$ to obtain,
since $\Phi(2p)=\Phi(-p)=0$,
\begin{align*}
\sum_{n=-p}^{2p-1} \Phi(n)e^{-2\pi i n\xi/3p}
&= 
\sum_{n=-p}^{2p-1}\Delta\Phi(n) \frac{e^{-2\pi i (n+1)\xi/3p}-e^{2\pi i \xi/3}}{e^{-2\pi i \xi/3p}-1}
\\
&= 
-(e^{-2\pi i \xi/3p}-1)^{-1}
\sum_{n=-p}^{2p-1}\Delta\Phi(n) (e^{-2\pi i (n+1)\xi/3p} - e^{2\pi i \xi/3})
\\
&= 
-(e^{-2\pi i \xi/3p}-1)^{-1}
e^{-2\pi i\xi/3p}
\sum_{n=-p}^{2p-1}
\Delta\Phi(n) e^{-2\pi i n\xi/3p}
\end{align*}
since $\sum_{n=-p}^{2p-1}(\Phi(n+1)-\Phi(n))=0$.
A second summation by parts yields a bound
\begin{equation}
\big|\sum_{n=-p}^{2p-1} \Phi(n)e^{-2\pi i n\xi/3p}\big|
\le C
\big| e^{-2\pi i \xi/3p}-1\big|^{-2}
\norm{\Delta^2\Phi}_{\ell^1}.
\end{equation}

Clearly 
$\norm{\Delta^2\Phi}_{\ell^1}\le Cp^{-1}$.
It is straightforward to verify that
\[
\norm{
| e^{-2\pi i \xi/3p}-1|^{-2}}_{\ell^1([0,p-1])}\le Cp^2.
\]
Combining these bounds yields the required inequality \eqref{eq:psiFT}.
\end{proof}

\subsection{Transference from $\integers^m$ to $\integers$.}
The next step is to transfer $\rho^\ddagger$ from $\integers^m$ to $\integers$.
Define $F:\integers^m\to\integers$ by
$F(k_1,\cdots,k_m) = \sum_{j=1}^m p^{j-1}k_j$.
$F$ maps $[0,p-1]^m$ bijectively to $[0,p^m-1]$.
For $n\in\integers$ define
$\rho^\dagger$ to be the pushforward of $\rho^\ddagger$ via $F$,
that is,
\begin{equation}
\rho^\dagger(n)=
\sum_{k: F(k)=n} \rho^\ddagger(k).
\end{equation}
Then
\begin{equation*}
\sum_{n\in\integers} e^{-2\pi i \theta n} \rho^\dagger(n)
= 
\sum_{k\in [-p,2p-1]^m} \rho^\ddagger(k) e^{-2\pi i \theta \sum_{j=1}^{m} p^{j-1}k_j}
= \widehat{\rho}(\theta,p\theta,p^2\theta,\cdots,p^{m-1}\theta),
\end{equation*}
so 
\begin{equation}
\norm{\widehat{\rho^\dagger}}_{L^{\infty}(\torus)}
\le \norm{\widehat{\rho^\ddagger}}_{L^\infty(\torus^m)} \le Cp^{-1/2}.
\end{equation}
The Fourier transform on the left-hand side is that for $\integers$; the one on the
right is that for $\integers^m$.
The same bound holds, of course, for any translate of $\rho^\dagger$.

We have defined a linear, positivity preserving 
operator $\Gamma:\ell^1(\integers_p^m)\to\ell^1(\integers)$;
$\Gamma$ extends a function to $\integers_{3p}^m$, multiplies by the cutoff function $\phi$,
transplants the result to $\integers^m$, then pushes it forward to $\integers$.
$\rho^\dagger= \Gamma(\sigma_{p,m})-\Gamma(\sigma^0_{p,m})$ is expressed
as $\mu^\dagger-\nu^\dagger$
where the summands have the following properties.
$\mu^\dagger$ is nonnegative and
\begin{equation}
\mu^\dagger \ge p^{-1}\sum_{j=-p}^{2p-1}\delta_{g(j)}
\end{equation}
where $g(j) = \sum_{r=1}^{m} p^{r-1}\modulo{j^r}{p}$,
while $\nu^\dagger$ is supported in $[-Cp^m, Cp^m]$
and satisfies 
\begin{equation}\label{supnormnubound}
\norm{\nu^\dagger}_\infty\le Cp^{-m}.\end{equation}
Finally, because $\Gamma$ is linear and preserves positivity,
$\phi\ge 0$, and $\phi\equiv 1$ on $[0,p-1]^m$, 
$\mu^\dagger \ge p^{-m}\sum_{n\in\tilde\scriptS}\delta_n$
where $\tilde\scriptS=\{g(j): j\in[0,p-1]\}$.

Now let $(p_k)$ be any sequence of odd primes 
satisfying \eqref{primehypotheses}, and
let $(a_k)$ be an arbitrary sequence of natural numbers
satisfying \eqref{firstajcondition} and \eqref{secondajcondition}.
Define $\lambda_k(n) = \rho^\dagger_{p,m}(n-a_k)$. 
$\lambda_k$ decomposes  as
$\lambda_k = \tilde\mu_k-\nu_k$
where $\tilde\mu_k(n)=\mu_k^\dagger(n-a_k)$
and $\nu_k(n) = \nu_k^\dagger(n-a_k)$.
The pair $\tilde\mu_k,\nu_k$ satisfies the hypotheses
of Theorem~\ref{thm:difference}; 
the maximal operator $f\mapsto \sup_k |f|*|\nu_k|$
is dominated by a constant multiple of the Hardy-Littlewood maximal operator
by virtue of \eqref{supnormnubound} and \eqref{secondajcondition}.
Therefore the maximal operator
$f\mapsto \sup_k |f*\tilde\mu_k|$ is of weak type $(1,1)$
on $\integers$.

Since each $\tilde\mu_k$ is nonnegative,
the same applies to $\sup_k |f*\mu_k^\star|$ for any sequence
of functions $0\le\mu_k^\star\le \tilde\mu_k$.
Since $\tilde\mu_k\ge p_k^{-m}\sum_{n\in\scriptS_k}\delta_n=\mu_k$,
we may set $\mu_k^\star=\mu_k$ to deduce Theorem~\ref{thm:main}.
\qed



\subsection{A second set of examples.}

The following variant produces examples which are less sparse, with 
$n_\nu\asymp \nu^m$ for $m=\frac{d+1}{d}$, for any positive integer $d$.
Fix a positive integer $d \ge 1$.
Let $(p_k: k\in\naturals)$ be any lacunary sequence of primes.
Let $(a_k)$ be an auxiliary sequence of natural numbers
satisfying 
\begin{equation}\label{variantakhypotheses}
\begin{split}
a_k &\le Cp_k^{d+1}
\\
a_{k+1} &\ge a_k + p_k^{d+1}
\end{split}
\end{equation}
Define
\begin{equation}
n(j,p) = j_1+p j_2 + \cdots + p^{d-1}j_d
+ p^d\modulo{|j|^2}{p}
\end{equation}
for $j\in [0,p-1]^d$
and
\begin{equation}
\scriptS_k=\{ n(j,p_k)+a_k: j\in[0,p-1]^d \}.
\end{equation}
Set $\scriptS=\cup_{k\in\naturals}\scriptS_k$,
and let the sequence $(n_\nu: \nu\in\naturals)$
be the elements of $\scriptS$, listed in increasing order.

\begin{theorem}\label{thm:variant}
Let the sequences of primes $(p_k)$ and natural numbers $(a_k)$ 
satisfy \eqref{primehypotheses} and\eqref{variantakhypotheses}.
Then the associated subsequence $(n_\nu)$ of $\naturals$ satisfies 
$n_\nu\asymp\nu^{(d+1)/d}$,
and the maximal function associated to $(n_\nu)$
is of weak type $(1,1)$ on $\integers$.
\end{theorem}

The proof of Theorem~\ref{thm:variant}
is essentially identical to that of Theorem~\ref{thm:main},
except that the exponential sum bound of Lemma~\ref{lemma:weil} is replaced
by the following simpler bound.
For $\xi\in \integers_p^{d+1}$ we write 
\[\xi = (\xi',\xi_{d+1}) \in \integers_p^d\times\integers_p.  \]

\begin{lemma} \label{lemma:quadraticexpsum}
Let $d$ be any positive integer, and let $p$ be any prime.
\begin{equation}
p^{-d} \Big|\sum_{n\in [0,p-1]^d}
e^{-2\pi i (n\cdot\xi'+ |n|^2\xi_{d+1})/p}\Big|
\le  p^{-d/2}
\ \ \text{ for all $0\ne\xi\in \integers_p^{d+1}$.} 
\end{equation}
\end{lemma}
\noindent
Here $|n|^2 = \modulo{\sum_{j=1}^d n_j^2}{p}$
where $n=(n_1,\cdots,n_d)$.

These exponential sums factor as products of $d$ Gauss sums,
so the lemma follows from the fact that
\begin{equation}
\sum_{n=0}^{p-1}
\big|e^{-2\pi i (an+bn^2)/p}\big|
= 
\begin{cases}
p^{1/2} &\text{ if } b\ne 0
\\
0 &\text{ if }  b=0 \text{ and } a\ne 0
\\
p & \text{ if } a=b=0.
\end{cases}
\end{equation}

Thus the measure 
\[\sigma = p^{-d}\sum_{n\in\integers_p^d}\delta_{(n,\modulo{|n|^2}{p})}\]
satisfies
$|\widehat{\sigma}(\xi)|\le p^{-d/2} \text{ whenever } \xi\ne 0$.
The proof of Theorem~\ref{thm:main} therefore applies. 
\qed

\begin{remark}
Theorem~\ref{thm:main} produces sequences satisfying 
$n_\nu\asymp \nu^m$ for $m=2,3,4,\cdots$.
For any prescribed rational exponent $r>2$,
an example satisfying $n_\nu\asymp \nu^r$ 
can be constructed by using one value of $m$
for some indices $k$ and a second value for the others;
details are left to the reader.
For any rational $r\in(1,2)$, an example may be constructed
in the same way using instead the construction of Theorem~\ref{thm:variant}.
\end{remark}


\begin{thebibliography}{20}

\bibitem{bourgain1}
J.~Bourgain,
{\em On the maximal ergodic theorem for certain subsets of the integers}, 
Israel J. Math. 61 (1988), no. 1, 39--72. 

\bibitem{bourgain2}
\bysame,
{\em On the pointwise ergodic theorem on $L\sp p$ for arithmetic sets}, 
Israel J. Math. 61 (1988), no. 1, 73--84. 

\bibitem{bourgain3}
\bysame,
{\em Pointwise ergodic theorems for arithmetic sets. 
With an appendix by the author, Harry Furstenberg, Yitzhak Katznelson and Donald S. Ornstein}, 
Inst. Hautes ´ 
Etudes Sci. Publ. Math. No. 69 (1989), 5–45.

\bibitem{bm1}
Z.~Buczolich and R.~D.~Mauldin,
{\em Divergent Square Averages}, preprint, 
math/0504067

\bibitem{bm2}
\bysame,
{\em Concepts behind divergent ergodic averages along squares}, 
Contemp. Math., 430 , Amer. Math. Soc., Providence, RI, 2007.  


\bibitem{christrough}
M.~Christ,
{\em Weak type $(1,1)$ bounds for rough operators},
Ann. of Math. (2) 128 (1988), no. 1, 19--42.


\bibitem{cfthesis}
C.~Fefferman,
{\em Inequalities for strongly singular convolution operators}, 
Acta Math. 124 1970 9--36.

\bibitem{patrick1}
P.~LaVictoire,
{\em An $L^1$ ergodic theorem for sparse random subsequences}, Math. Research Letters,
to appear. 
arXiv:0812.3175 

\bibitem{patrickbm}
\bysame,
{\em Universally $L^1$--bad arithmetic sequences},
preprint, arXiv:0905.3865.

\bibitem{LN}
R.~Lidl and H.~Niederreiter,
{\em Finite Fields},
Cambridge University Press, 1997.

\bibitem{rudin}
W.~Rudin,
{\em Trigonometric series with gaps},
J. Math. Mech. 9 1960 203--227.

\bibitem{uz}
R.~Urban and J.~Zienkiewicz,
{\em Weak type $(1,1)$ estimates for a class of discrete rough maximal functions},
Math. Res. Lett. 14 (2007), no. 2, 227--237.


\end{thebibliography}
\end{document}